\documentclass[10pt]{amsart}
\usepackage{amsmath}
\usepackage{amsthm}
\usepackage{color}
\usepackage{amscd}
\usepackage{amsfonts}
\usepackage{graphicx}
\usepackage{epsfig}
\usepackage{amssymb,latexsym}
\usepackage{bm}

\newcommand{\supp}{\operatorname{supp}}

\newcommand{\ndv}{\operatorname{Ndiv}}
\def\bbu{{\boldsymbol{u}}}

\def\bba{{\boldsymbol{a}}}

\def\bbD{{\boldsymbol{D}}}
\def\bbX{{\boldsymbol{X}}}

\def\bu{\mathbf{u}}

\def\bU{\mathbf{U}}
\def\bbx{{\boldsymbol{x}}}
\def\bx{\mathbf{x}}

\def\bby{{\boldsymbol{y}}}

\def\R{\mathbb{R}}

\def\bF{\mathbf F}
\def\bM{\mathbf M}
\def\cM{\mathcal M}
\def\bP{\mathbf P}
\def\cP{\mathcal P}
\def\bbP{\mathbb P}

\def\nB{\mathbf B}
\def\cC{\mathcal C}
\def\uno{\mathbf{1}}

\numberwithin{equation}{section}

\newtheorem{theorem}{Theorem}[section]
\newtheorem{corollary}[theorem]{Corollary}
\newtheorem{definition}{Definition}[section]

\newtheorem{problem}[theorem]{Problem}

\begin{document}

\title{On non-locality in the Calculus of Variations}
\author{Pablo Pedregal}
\date{} 
\thanks{INEI, U. de Castilla-La Mancha, 13071 Ciudad Real, SPAIN. Supported by grant 
MTM2017-83740-P}
\begin{abstract}
Non-locality is being intensively studied in various PDE-contexts and in variational problems. The numerical approximation also looks challenging, as well as the application of these models to Continuum Mechanics and Image Analysis, among other areas. Even though there is a growing body of deep and fundamental knowledge about non-locality, for variational principles there are still very basic questions that have not been addressed so far. Taking some of these as a motivation, we describe a general perspective on distinct classes of non-local variational principles setting a program for the analysis of this kind of problems. We start such program with the simplest problem possible: that of scalar, uni-dimensional cases, under a particular class of non-locality. Even in this simple initial scenario, one finds quite unexpected facts to the point that our intuition about local, classic problems can no longer guide us for these new problems.
There are three main issues worth highlighting, in the particular situation treated:
\begin{enumerate}
\item natural underlying spaces involve different non-local types of derivatives as, for instance, fractional Sobolev spaces;
\item no convexity of integrands is required for existence of minimizers;
\item optimality is formulated in terms of quite special integral equations rather than differential equations.
\end{enumerate}
We are thus able to provide some specific answers to the initial questions that motivated our investigation. In subsequent papers, we will move on to consider the higher dimensional situation driven by the possibility that no convexity or quasiconvexity might be involved in weak lower semicontinuity in a full vector, higher dimensional situation. 
\end{abstract}
\maketitle
\section{Introduction}
Non-locality is a hot topic these days both in PDE, and in variational problems, as well as in Continuum Mechanics and Elasticity. The motivation, the ideas, the techniques cover a huge spectrum of material hard to describe in a few paragraphs. In particular, Peridynamics has emerged as a main body of ideas of interest in the Theory of Elasticity. A lot has been written  about non-locality in Analysis and applications, and yet it looks as if some of the most basic issues still require some attention. 

To realize how far we are from understanding even the simplest of situations and how nothing we take for granted in the local case can be translated in a trivial form to this non-local scenario, we will focus on the following innocent-looking problem.
\begin{problem}\label{primero}
Consider the functional
$$
E_p(u)=\int_0^1\int_0^1\left|\frac{u(x)-u(y)}{x-y}\right|^p\,dx\,dy
$$
for competing functions $u$ in $L^p(0, 1)$. We assume first $p>2$. 
If nothing else is demanded of feasible functions, then constant functions are minimizers.  However, we will check that functions  $u\in L^p(0, 1)$ for which $E_p(u)<+\infty$, admit end-point conditions because those functions can be shown to be H\"older continuous. It is legitimate, then, to look for minimizers of $E_p(u)$ among those functions $u\in L^p(0, 1)$ complying with, say,
$$
u(0)=0, \quad u(1)=1.
$$
Three basic issues require a precise answer:
\begin{enumerate}
\item are there minimizers for such a problem?
\item if so, is the linear function $u(x)=x$ a minimizer of the problem, or even the unique minimizer?
\item what is the form of optimality conditions for such a variational problem?
\end{enumerate}
One would be tempted to let it go led by the corresponding local case in which one tries to minimize
$$
I_p(u)=\int_0^1 u'(x)^p\,dx
$$
under the same end-point conditions. It is elementary to argue that in this case the linear function $u(x)=x$ is the unique minimizer. However, there are some unexpected facts for the non-local version above. 

For the case $1\le p\le 2$, functions in $L^p(0, 1)$ with finite energy $E_p<\infty$ need not be continuous, and hence end-point constraint cannot be imposed to begin with. We use, however, the case $p=2$ for some numerical experiments, to facilitate the implementation. 
\end{problem}
The central role played by convexity for classic variational principles is something very well established to the point that the lack of this structural condition leads in many situations to lack of minimizers. Possibly, the simplest examples are the one-dimensional versions of two-well Bolza problems. 
\begin{problem}\label{nlbolza}
The variational problem
$$
I(u)=\int_0^1\left[\frac14(u'(x)^2-1)^2+\frac12u(x)^2\right]\,dx
$$
under vanishing end-point conditions lacks minimizers. Minimizing sequences are of the form of saw-tooth functions with slopes $\pm1$ refining its teeth without limit. 
The non-local version would be
$$
E(u)=\int_0^1\int_0^1\left[\frac14\left(\left(\frac{u(y)-u(x)}{y-x}\right)^2-1\right)^2+\frac12u(x)^2\right]\,dy\,dx,
$$
under the same end-point conditions. Is it true, as in the local version, that there are no minimizer for this non-local problem? Again one would be tempted to support that this is so, and once again one would face a surprise. In fact, one can also think about the variant
$$
\overline E(u)=\int_0^1\int_0^1\left[\frac14\left(\left(\frac{u(y)-u(x)}{y-x}\right)^2-1\right)^2\right]\,dy\,dx
$$
without the lower-order term. 
This time, the local version admit infinitely many minimizers, but it is not clear if all of those would be minimizers for this non-local version. Note that these examples have growth of order $4>2$.
\end{problem}

We aim at starting the systematic study of this kind of variational problems for which we would like to be able to answer very specific and concrete questions, in addition to exploring all the related functional analytical framework and its potential applicability to other areas of research. In this initial contribution, further to describing our main general motivation, we will take our ability to provide specific answers to the two previous problems as a measure of success. 

Non-local variational problems have undergone an unprecedented raise in interest, perhaps pushed by non-local theories in Continuum Mechanics. Though these are not new (see \cite{eringen} for instance), they have been revived by the more recent theory of Peridynamics (\cite{silling}, \cite{sillingetal}). At the more mathematical level, non-local variational problems were started to be considered even before Peridynamics (\cite{branro}, \cite{pedregal1}), and a lot of work in various different directions has been performed since then. Another area where non-local functionals have been considered systematically is that of imaging models and free discontinuity problems where a search of new ways to approximate difficult local functionals by non-local ones has been pursued (\cite{braides}, \cite{cortesani}). 

We can hardly mention all papers that have contributed to these areas. Note that even more works deal with non-local theories of PDEs, though this field is not of concern here. We just mention a bunch of representative contributions in various topics dealing with non-locality in variational problems:
\begin{itemize}
\item Fractional and non-local theories in elasticity, and its relationship to local models: \cite{bellidocuetomora1}, \cite{mengeshadu2}.
\item Mathematical analysis of non-local variational principles: \cite{bellidomora1}, \cite{bellidomora2}.
\item Convergence of non-local models to their local counterparts: \cite{bellidocuetomora2}, \cite{bellidomorapedregal}.
\item Relaxation and related issues: \cite{kreisbeckzappale1}, \cite{kreisbeckzappale2}, \cite{moratellini}.
\item Non-local spaces of functions: \cite{bourgainbrezismironescu}, \cite{brezis2}, \cite{dinezzapalatuccivaldinoci}, \cite{ponce1}, \cite{ponce2}, \cite{silhavy}.
\item One-dimensional problems: \cite{brezis1}, \cite{lussardi1}.
\item Image and free discontinuity models: in addition to those already cited \cite{boulanger}, \cite{brezis3}, \cite{lussardi}.
\item Non-locality in other areas: \cite{anza}, \cite{gobbino}.
\end{itemize}
So far, the family of non-local variational problems that have been considered are of the general form
\begin{equation}\label{vpnl}
E(u)=\int_{\Omega\times\Omega} W(\bbx, \bby, u(\bbx), u(\bby))\,d\bby\,d\bbx,
\end{equation}
and the central issue of weak lower semicontinuity, as a main ingredient for the direct method of the Calculus of Variations, has been studied only with respect to weak convergence for feasible functions or fields $\bbu$. This has led to some important results and some new notions of (non-local) convexity (\cite{bellidomora2}, \cite{elbau}, \cite{pedregal1}). However, no specific variational problem has been examined from the viewpoint of existence of minimizers, in part because Lebesgue spaces where this analysis has been carried out do not allow for boundary values to be assigned directly. This is also one main trouble with variational problems over fractional Sobolev spaces (\cite{dinezzapalatuccivaldinoci}) where, typically, boundary conditions are imposed by demanding that feasible functions are identical to a preassigned function off the domain $\Omega$, or at least in a non-negligible strip around $\partial\Omega$ (\cite{bellidocuetomora1}). Apparently, the use of fractional Sobolev spaces in variational problems over bounded domains still need some new ideas. In this context of fractional Sobolev spaces, the so-called fractional gradient has been considered and extensively studied, together with parallel central results with respect to its local counterpart. Check \cite{ponce3}, \cite{shiehspector}, \cite{silhavy}. Variational principles explicitly depending on the fractional gradient have been considered (\cite{shiehspector}), even in a vector setting involving the property of polyconvexity (\cite{bellidocuetomora1}). 

Going back to problems of the form \eqref{vpnl}, two important topics have been considered in greater detail: relaxation of this non-local variational problems, and convergence to local theories when the horizon parameter of the non-local interaction is sent to zero. The analysis of the first has shown some unexpected results with no parallelism in local problems, as sometimes relaxation takes the problem outside the natural family of variational principles (\cite{kreisbeckzappale2}, \cite{moratellini}); the convergence in the latter has led to some significant limit facts (\cite{bellidocuetomora2}, \cite{bellidomorapedregal}). 

Despite all of these deep developments, there is no explicit example, even very simple cases as the ones stated in Problems \ref{primero} and \ref{nlbolza}, where basic questions have been answered. One point we would like to stress is that even if one starts in a big space for a non-local variational problem (like a Lebesgue space), the class of functions for which the functional takes on finite values may be a much more restrictive family of more regular functions. This is trivial when the integrand in the functional depends explicitly on the weak gradient, but it is not so clear, a priori, if there is no explicit gradient dependence. This is one natural reason of why weak lower semicontinuity was started to be studied in Lebesgue spaces, rather than on more restrictive spaces of functions. 

On the other hand, we would like to introduce some formalism to somehow classify non-local variational principles of various kinds (Section \ref{dos}). In particular, we set here a whole program to undertake the understanding of such non-local variational principles in their fundamental questions. We select one of those frameworks, and start with such a program for the simplest case possible: that of scalar, one-dimensional problems. More specifically:
\begin{enumerate}
\item Section \ref{tres}: we focus on the natural, underlying spaces to appropriately setup this sort of non-local variational problems. Though these spaces turn out to be, in the one-dimensional setting, the standard fractional Sobolev spaces, the variational problems themselves are quite different from the local classical ones.
\item Those new, non-local variational problems are studied from the point of view of the direct method in Section \ref{cuatro}, establishing a basic weak lower semicontinuity result, and, as a consequence, a typical existence theorem. It is remarkable that no convexity whatsoever is required.
\item Section \ref{cinco}. Optimality is explored in this section. Quite surprisingly, it can be formulated in terms of some special integral equations.
\item In Section \ref{seis}, we spend some more time analyzing such integral equations and their solutions in some easy examples to gain some intuition. 
\item In the scalar, one-dimensional situation, simple approximations of optimal solutions under convexity, can be performed. In particular, we will see an approximated profile of the optimal solution for Problem \ref{primero}. 
\end{enumerate}

As a result of our investigation in these sections, we are able to provide an answer to Problems \ref{primero} and \ref{nlbolza}. Concerning the first, we can say that there are minimizers; in fact, due to strict convexity, there is a unique such minimizer, but it is not the linear function $u(x)=x$. This can be easily checked through optimality conditions that, as indicated above, come in the form of some integral equation: as usual, given a functional equation, it may be easy or doable to check if a given function is or is not a solution; it may be impossible to find the solution. What is a bit shocking is that there is no convexity requirement involved for the existence of minimizers: for every continuous, coercive integrand there are minimizers !! In particular, there are such optimal solutions for the non-local version of the two-well Bolza problem considered in Problem \ref{nlbolza}. 

Our results here for the scalar, one-dimensional situation are just the starting point to proceeding to the higher dimensional case, or even the vector case. We will do so in forthcoming contributions.

\section{General overview}\label{dos}
Let us start form the well-known local case in which our funcional is of integral-type
$$
I(\bbu)=\int_\Omega W(\bbx, \bbu(\bbx), \nabla\bbu(\bbx))\,d\bbx
$$
where
$$
W(\bbx, \bbu, \bF):\Omega\times\R^n\times\R^{n\times N}\to\R
$$
is a suitable integrand, and $\Omega\subset\R^N$ is a bounded, regular domain. This functional can be interpreted in many different ways depending on the context where modeling is pursued. In hyperelasticity, for example, it may be a way to measure the energy associated with deformations in such a way that global minimizers would correspond to stable equilibrium configurations. For the sake of simplicity, we will omit the $(\bbx, \bbu)$ dependence as it is not relevant for what we are about to say, and write instead
$$
I(\bbu)=\int_\Omega W(\nabla\bbu(\bbx))\,d\bbx.
$$
It is well established that the property of quasiconvexity of $W(\bF)$ is a necessary and sufficient condition for the weak lower semicontinuity of $I$ over typical Sobolev spaces (\cite{dacorogna}, \cite{rindler}), which in turn is one of the two main ingredients for the direct method of the Calculus of Variations. When this property does not hold, then non-existence of minimizers may occur, and the analysis follows by exploring relaxation. 

One general way to express the passage from a functional like $I(\bbu)$ to its relaxed version involves the use of gradient Young measures (\cite{pedregal}, \cite{rindler}) to write
\begin{equation}\label{relaxation}
\overline I(\bbu)=\int_\Omega\int_{\R^{n\times N}}W(\bF)\,d\nu_{\bbx, \bbu}(\bF)\,d\bbx,
\end{equation}
where 
$$
\nu_\bbu=\{\nu_{\bbu, \bbx}\}_{\bbx\in\Omega},\quad \supp\nu_{\bbx, \bbu}\subset\R^{n\times N},
$$ 
is a family of probability measures, one for each $\bbx\in\Omega$, referred to as the associated gradient Young measure. Such family of probability measures generated by relaxation encodes the information to build minimizing sequences for the original problem. In addition to enjoying fundamental properties not fully yet understood, we also have 
$$
\nabla\bbu(\bbx)=\int_\bM \bF\,d\nu_{\bbu, \bbx}(\bF).
$$
It is not our objective, nor is the appropriate place, to discuss further this issue. Our aim is to focus on \eqref{relaxation} as a way to define classes of non-local functionals by selecting rules to determine the family of probability measures
$$
(\bbx, \bbu)\mapsto \nu_{\bbx, \bbu}.
$$
\begin{definition}\label{abstracta}
For a bounded, regular domain $\Omega\subset\R^N$, consider a mapping 
$$
\mu=\mu_{\bx, \bbu}: \Omega\times\cM(\Omega; \R^n)\mapsto\cP(\R^{n\times N})
$$
where $\cM(\Omega; \R^n)$ designates the class of measurable functions in $\Omega$ taking values in $\R^n$, and $\cP(\R^{n\times N})$ stands for the set of Borel probability measures supported in $\R^{n\times N}$. We say that such a mapping generates the family of variational problems corresponding to functionals
$$
I:\cM(\Omega; \R^n)\to\R,\quad I(\bbu)=\int_\Omega \int_{\R^{n\times N}}W(\bbx, \bbu(\bbx), \bF)\,d\mu_{\bbx, \bbu}(\bF)\,d\bbx,
$$
for Carath\'eodory integrands
$$
W(\bbx, \bbu, \bF):\Omega\times\R^n\times\R^{n\times N}\to\R
$$
which are measurable in $\bbx$ and continuous in $(\bbu, \bF)$, provided all the maps
$$
\bbx\mapsto \int_{\R^{n\times N}}W(\bbx, \bbu(\bbx), \bF)\,d\mu_{\bbx, \bbu}(\bF)
$$
are measurable.
For each given $\bbu\in\cM(\Omega; \R^n)$, the mapping
$$
\bbD_\mu \bbu(\bbx):\Omega\mapsto\mu_{\bbx, \bbu}\in\cP(\R^{n\times N})
$$
is called the corresponding non-local gradient for $\bbu$.
Particular rules may require more restrictions on functions $\bbu$ than just measurability. 
\end{definition}
Let us remind readers that the most straightforward way to define probability measures in $\cP(\R^{n\times N})$ consists of determining its action on continuous functions (with a vanishing limit at infinity)
$$
\langle \Phi, \mu\rangle=\int_{\R^{n\times N}}\Phi(\bF)\,d\mu(\bF),
$$
and one of the most efficient ways to define such probability measures proceeds  through the standard process of pushing forward with suitable maps; namely, if $(\bbP, \Sigma, \pi)$ is a probability space and
$$
\Psi(\bbX):\bbP\to\R^{n\times N}
$$
is a measurable mapping, then the push-forward $\Psi_*(\pi)$ of $\pi$ on to $\R^{n\times N}$ is the probability measure supported in $\R^{n\times N}$ defined through 
$$
\langle\Phi, \Psi_*(\pi)\rangle=\langle\Phi(\Psi), \pi\rangle.
$$
We will be using this procedure in most examples without further notice. 

We consider several initial such rules to generate classes of non-local variational problems, and then focus on the one we would like to concentrate our analysis on here. The rule above that has motivated this concept is not a true instance because underlying gradient Young measures come from relaxation and cannot be associated with each $\bbu$ with no reference to additional ingredients. In fact, they are chosen by minimizing an already-existing functional. 

\begin{enumerate}
\item The trivial case corresponds to local, classical variational principles for Sobolev functions
$$
\mu_{\bbx, \bbu}=\delta_{\nabla \bbu(\bbx)}(\bF),\quad \langle\Phi, \mu_{\bbx, \bbu}\rangle=\Phi(\nabla \bbu(\bbx)).
$$
The corresponding gradient is just the usual weak gradient for Sobolev functions.
\item The fractional case
$$
\langle\Phi, \mu_{\bbx, \bbu}\rangle=\int_\Omega\Phi\left(\frac{\bbu(\bby)-\bbu(\bbx)}{|\bby-\bbx|^\alpha}\otimes\frac{\bby-\bbx}{|\bby-\bbx|}\right)\,d\bby
$$
for an appropriate exponent $\alpha$. The associated non-local gradient would be
the probability measure
$$
\bbD\bbu(\bbx)=\frac1{|\Omega|}\frac{\bbu(\bby)-\bbu(\bbx)}{|\bby-\bbx|^\alpha}\otimes\frac{\bby-\bbx}{|\bby-\bbx|}\,\left.d\bby\right|_\Omega.
$$
\item The gradient, average case
$$
\langle\Phi, \mu_{\bbx, \bbu}\rangle=\int_\bbP\Phi\left(\frac1{V(\bP(\bbx, \bbX))}\int_{\bP(\bbx, \bbX)}\nabla\bbu(\bby)\,d\bby\right)\,d\bbX
$$
where $\bbX\in \bbP$, and $\bbP$ is a probability space of parameters, each of which, together with $\bbx\in\Omega$, determines a measurable subset
$$
\bP(\bbx, \bbX)\subset\Omega
$$
with $N$-dimensional measure $V(\bP(\bbx, \bbX))$, where to perform the average of the gradient of $\bbu$. The obvious case is
$$
\langle\Phi, \mu_{\bbx, \bbu}\rangle=\int_0^H\Phi\left(\frac1{V(\nB(\bbx, r))}\int_{\nB(\bbx, r)}\nabla\bbu(\bby)\,d\bby\right)\,dr,
$$
where $H>0$ would be the ``horizon" of the non-locality. Balls are understood intersected with $\Omega$. In this situation, non-local gradients are
$$
\bbD\bbu(\bbx)=\frac1{V(\bP(\bbx, \bbX))}\int_{\bP(\bbx, \bbX)}\nabla\bbu(\bby)\,d\bby\,d\bbX.
$$
\item The mean rule. For every mapping $\mu$ as in Definition \ref{abstracta}, we can consider its mean rule $\overline\mu$, which is another form of non-locality, namely
$$
\overline\mu_{\bx, \bbu}:\Omega\times\cM(\Omega; \R^n)\mapsto\cP(\R^{n\times N})
$$
and
$$
\langle\overline\mu_{\bx, \bbu}, \Phi\rangle=\Phi\left(\int_{\R^{n\times N}}\bF\,d\mu_{\bx, \bbu}(\bF)\right).
$$
In compact form, we can write
$$
\overline\mu_{\bbx, \bbu}=\delta_{\bM_1(\bbx, \bbu)}(\bF)
$$
where
$$
\bM_1(\bbx, \bbu)=\int_{\R^{n\times N}}\bF\,d\mu_{\bx, \bbu}(\bF)
$$
is the first moment of $\mu_{\bbx, \bbu}$, and $\delta$ is the Dirac mass.
The corresponding non-local gradient for $\overline\mu$ is just the average of the non-local gradient of $\mu$, i.e.
$$
\bbD_{\overline \mu}\bbu(\bbx)=\bM_1(\bbx, \bbu).
$$
Note the difference between the variational principles associated with $\mu_{\bbx, \bbu}$ and with its mean $\overline\mu_{\bbx, \bbu}$
\begin{align}
I(\bbu)=&\int_\Omega \int_{\R^{n\times N}}W(\bbx, \bbu(\bbx), \bF)\,d\mu_{\bbx, \bbu}(\bF)\,d\bbx,\nonumber\\
\overline I(\bbu)=&\int_\Omega W\left(\bbx, \bbu(\bbx), \int_{\R^{n\times N}}\bF\,\mu_{\bbx, \bbu}(\bF)\,d\bF\right)\,d\bbx\nonumber\\
=&\int_\Omega W\left(\bbx, \bbu(\bbx), \bbD_{\overline\mu}(\bbx)\right)\,d\bbx.\nonumber
\end{align} 
\end{enumerate}

\subsection{One special class of non-locality} 
We would like to focus, however, on a different type of non-locality motivated by its potential interpretation in the context of hyper-elasticity, though we remain at a purely mathematical level at this stage.
Our basic postulate is the assumption that the internal energy $E(\bbu)$ associated with a deformation of a body
$$
\bbu(\bbx):D\subset\R^N\to\R^N,
$$
where $D$ is some selected, unit reference domain in $\R^N$, is measured with a density $W$ acting on the basic building blocks for deformations, which are taken to be the affine maps from $\R^N$ to $\R^N$. We know that the linear part of these are identified, once a basis of $\R^N$ has been chosen, with $N\times N$-matrices $\bF$. Therefore we postulate that the internal energy is translation-invariant, that the main variables for $W$ are $N\times N$-matrices, and 
\begin{equation}\label{postulado}
W(\bF):\R^{N\times N}\to\R,\quad W(\bF)=E(\bbu_\bF),
\end{equation}
when we take 
\begin{equation}\label{afin}
\bbu_\bF(\bbx)=\bba+\bF\bbx,\quad \bbx\in D, \quad \bba\in\R^N.
\end{equation}
From here, and realizing that affine deformations are characterized by $\nabla \bbu(\bbx)=\bF$, one proceeds with the standard local theory in which the internal energy associated with a general deformation $\bu(\bbx)$ is taken to be
$$
E(\bbu)=\int_\Omega W(\nabla\bbu(\bbx))\,d\bbx.
$$
Affine deformations in \eqref{afin} and their linear parts $\bF$ are also generically characterized, in a unique way, as being generated by the images of $N+1$ generic points 
$$
\bbx_0, \bbx_1, \dots, \bbx_N\in D
$$
and their images
$$
\bbu_\bF(\bbx_0), \bbu_\bF(\bbx_1),\dots, \bbu_\bF(\bbx_N)\in\R^N,
$$
that is to say
$$
\bF=\begin{pmatrix}\bbu_\bF(\bbx_1)-\bbu_\bF(\bbx_0)&\dots&\bbu_\bF(\bbx_N)-\bbu_\bF(\bbx_0)\end{pmatrix}
\begin{pmatrix}\bbx_1-\bbx_0\dots&\bbx_N-\bbx_0\end{pmatrix}^{-1}.
$$
This last formula is trivial, but it yields, when the affine deformation $\bbu_\bF$ is replaced by any feasible $\bbu$, a non-local way to measure the internal energy $E(\bbu)$ through the multiple integral
\begin{align}
\int_{\Omega^{N+1}} W\left(\begin{pmatrix}\bbu(\bbx_1)-\bbu(\bbx_0)\dots\bbu(\bbx_N)-\bbu(\bbx_0)\end{pmatrix}
\begin{pmatrix}\bbx_1-\bbx_0\dots\bbx_N-\bbx_0\end{pmatrix}^{-1}\right)&\times\nonumber\\
\times\,d\bbx_N\dots d\bbx_1\,d\bbx_0.&\nonumber
\end{align}
Both ways are consistent for the affine deformation $\bbu_\bF$ (provided $|\Omega|=1$). 

To simplify notation put
\begin{gather}
\bbX=\begin{pmatrix}\bbx_1&\dots&\bbx_N\end{pmatrix}\in\R^{N\times N},\quad \bbx=\bbx_0,\quad \uno=(1, \dots, 1)\in\R^N,\nonumber\\
\bbx\otimes\uno=\begin{pmatrix}\bbx&\dots&\bbx\end{pmatrix}\in\R^{N\times N},\nonumber
\end{gather}
and then
\begin{gather}
\bbu(\bbX)=\begin{pmatrix}\bbu(\bbx_1)& \bbu(\bbx_2)&\dots,&\bbu(\bbx_N)\end{pmatrix}\in\R^{N\times N},\nonumber\\
\bbu(\bbx, \bbX)=\bbu(\bbX)-\bbu(\bbx)\otimes\uno\in\R^{N\times N},\nonumber\\
\bbD\bbu(\bbx, \bbX)=(\bbu(\bbX)-\bbu(\bbx)\otimes\uno)(\bbX-\bbx\otimes\uno)^{-1}.\nonumber
\end{gather}
Our way to measure internal energy in a non-local way is written in the compact form
\begin{align}
E(\bbu)=&\int_\Omega\int_{\Omega^N} W(\bbu(\bbx, \bbX)(\bbX-\bbx\otimes\uno)^{-1})\,d\bbX\,d\bbx\nonumber\\
=&\int_\Omega\int_{\Omega^N} W((\bbu(\bbX)-\bbu(\bbx)\otimes\uno)(\bbX-\bbx\otimes\uno)^{-1})\,d\bbX\,d\bbx\nonumber\\
=&\int_\Omega\int_{\Omega^N} W(\bbD\bbu(\bbx, \bbX))\,d\bbX\,d\bbx.\nonumber
\end{align}
This corresponds exactly to the rule, in the context of Definition \ref{abstracta}, 
$$
\langle\Phi, \mu_{\bbx, \bbu}\rangle=\int_{\Omega^N}\Phi(\bbD \bbu(\bbx, \bbX))\,d\bbX.
$$
From here, it is easy to generalize it to incorporate other dependencies by putting
\begin{equation}\label{formageneral}
E(\bbu)=\int_\Omega\int_{\Omega^N} W(\bbx, \bbu(\bbx), (\bbu(\bbX)-\bbu(\bbx)\otimes\uno)(\bbX-\bbx\otimes\uno)^{-1})\,d\bbX\,d\bbx,
\end{equation}
or, in compact form,
\begin{equation}\label{formageneralcompacta}
E(\bbu)=\int_\Omega\int_{\Omega^N} W(\bbx, \bbu(\bbx), \bbD\bbu(\bbx, \bbX))\,d\bbX\,d\bbx.
\end{equation}
The functional we have written in \eqref{formageneral} is a general vector problem for a density
$$
W(\bbx, \bbu, \bF):\Omega\times\R^N\times\R^{N\times N}\to\R,
$$
and competing mappings 
$$
\bbu(\bbx):\Omega\subset\R^N\to\R^N.
$$ 
Nothing keeps us from considering the general situation in which
$$
W(\bbx, \bbu, \bF):\Omega\times\R^n\times\R^{n\times N}\to\R,
$$
for feasible mappings 
$$
\bbu(\bbx):\Omega\subset\R^N\to\R^n,
$$ 
where dimension $n$ could be different from $N$. In particular, the case $n=1$ 
\begin{align}
E(u)=&\int_\Omega\int_{\Omega^N} W(\bbx, u(\bbx), (u(\bbX)-u(\bbx)\uno)(\bbX-\bbx\otimes\uno)^{-1})\,d\bbX\,d\bbx\nonumber\\
=&\int_\Omega\int_{\Omega^N} W(\bbx, u(\bbx), \bbD u(\bbx, \bbX))\,d\bbX\,d\bbx\nonumber
\end{align}
will be referred to as the scalar case. It is not difficult to envision more general ingredients that can be added to this raw model, like implementing a horizon parameter $\delta$ to tame the range of non-local interactions. 

Our intention here is to start the mathematical analysis of this kind of non-local variational problems. Nothing will be claimed at this stage from the mechanical point of view. 

\subsection{Program}
As usual, the fundamental steps we would like to start covering concerning these non-local variational problems can be organized in the following way:
\begin{enumerate}
\item Natural spaces of functions where non-local functionals are well-defined.
\item Structural hypotheses on integrands to guarantee some suitable weak-lower semicontinuity.
\item Existence theorems. 
\item Optimality conditions.
\item Relaxation, if applicable.
\end{enumerate}
On the other hand, one would proceed covering:
\begin{enumerate}
\item Scalar, one-dimensional problems: $n=N=1$.
\item Scalar, higher-dimensional problems: $n=1$, $N>1$. 
\item Vector problems: $n, N>1$.
\end{enumerate}
It is a program to fully understand such family of variational problems. In this initial contribution, we will be contented dealing with the scalar, one-dimensional problem as a way to anticipate unexpected facts, difficulties, places where emphasis is recommended, etc. In particular, to measure success in this regard, we seek to provide as complete an answer as possible to Problems \ref{primero} and \ref{nlbolza}. 

\section{Spaces}\label{tres}
Each family of non-local problems gives rise to its own collection of natural functional spaces by demanding that all functions
\begin{equation}\label{pspace}
\bbx\in\Omega\mapsto \langle |\cdot|^p, \mu_{\bbx, u}\rangle
\end{equation}
belong to $L^p(\Omega)$ for  
$u\in L^p(\Omega)$, 
and $p\in[1, \infty]$. We are talking about the following collection of functions
\begin{equation}\label{ppspace}
\left\{u\in L^p(\Omega); \int_\Omega\int_{\R^N}|\bF|^p\,d\mu_{\bbx, u}(\bF)\,d\bbx<\infty\right\}.
\end{equation}
Let us examine, for the sake of illustration, some of the initial situations in the last section. 
\begin{enumerate}
\item For the classical local case, natural spaces are, of course, the standard Sobolev spaces $W^{1, p}(\Omega)$. There is nothing else to say.
\item For the fractional case, we are concerned about functions $u\in L^p(\Omega)$ such that
$$
\int_{\Omega\times\Omega}\frac{|u(\bby)-u(\bbx)|^p}{|\bby-\bbx|^{\alpha p}}\,d\bby\,d\bbx<+\infty.
$$
For appropriate exponents $\alpha$, these are the fractional Sobolev spaces that are being extensively studied these days. We have already commented about this in the Introduction. 
\item For the gradient, average situation we must be concerned about functions $u\in L^p(\Omega)$ for which 
$$
\int_\Omega\int_P\frac1{V(\bP(\bbx, \bbX))^p}\left|\int_{\bP(\bbx, \bbX)}\nabla u(\bby)\,d\bby\right|^p\,d\bbX\,d\bbx<\infty.
$$
As far as we can tell, these family of functions have not yet been examined. 
\item As in the previous section, for each mapping $\mu$ and its corresponding space based on \eqref{pspace}, there is a corresponding space changing \eqref{pspace} to
$$
\bbx\in\Omega\mapsto \left|\langle \bF, \mu_{\bbx, u}\rangle\right|^p,
$$
and \eqref{ppspace} to
$$
\left\{u\in L^p(\Omega); \int_\Omega\left|\int_{\R^N}\bF\,d\mu_{\bbx, u}(\bF)\right|^p\,d\bbx<\infty\right\}.
$$

\end{enumerate}
The family of spaces that we would like to consider, from the perspective of the non-local variational problems that we want to examine, are
$$
NW^{1, p}(\Omega)=\{u\in L^p(\Omega): \bbD u(\bbx, \bbX)\in L^p(\Omega\times\Omega^N; \R^N)\}.
$$
One starting point would be to study the relationship of such space to the 
 standard Sobolev space $W^{1, p}(\Omega)$, especially in sight of results in \cite{bourgainbrezismironescu}, and other similar articles. But, given that we do not have any initial intuition on the corresponding family of non-local variational problems, we begin by exploring the one-dimensional situation $N=1$. In this case
$$
\bbD u(x, X)=\frac{u(X)-u(x)}{X-x}.
$$
It looks reasonable to consider the space
$$
NW^{1, p}(0, 1)=\{u\in L^p(0, 1):\bbD u(x, X)\in L^p((0, 1)^2)\},
$$
for an exponent $p\in[1, \infty)$, and
$$
NW^{1, \infty}(0, 1)=\{u\in L^\infty(0, 1):\bbD u(x, X)\in L^\infty((0, 1)^2)\}.
$$
The natural norm in these spaces is
\begin{equation}\label{naturalnorm}
\|u\|_{NW^{1, p}(0, 1)}\equiv\|u\|_{L^p(0, 1)}+\|\bbD u\|_{L^p((0, 1)^2)}
\end{equation}
for all $p$. The case $p=2$ corresponds to a inner product
$$
\langle u, v\rangle=\int_0^1 u(x)v(x)\,dx+\int_{(0, 1)^2}\bbD u(x, X)\bbD v(x, X)\,dX\,dx.
$$
We put $NH^1(0, 1)$ to mean $NW^{1, 2}(0, 1)$. 

In this one-dimensional situation, we recognize that these spaces are the standard fractional Sobolev spaces (\cite{bourgainbrezismironescu}, \cite{dinezzapalatuccivaldinoci}) for 
$$
s=1-1/p,\quad 1<p<\infty.
$$
We will, however, keep the notation $NW^{1, p}(0, 1)$ to be consistent with the higher dimensional case, which will be addressed in a forthcoming work. As far as we can tell, these spaces in the higher dimensional situation have not been considered yet.

As a consequence of the fact 
$$
NW^{1, p}(0, 1)=W^{1/q, p}(0, 1),\quad \frac1p+\frac1q=1,
$$ 
we have a lot of fundamental results at our disposal. We focus especially on two of them taken directly from \cite{dinezzapalatuccivaldinoci}. We only need here the one-dimensional versions.
\begin{theorem}[Theorem 7.1, \cite{dinezzapalatuccivaldinoci}]\label{uno}
Every bounded set in $NW^p(0, 1)$ is precompact in $L^p(0, 1)$.
\end{theorem}

 In particular, we would like to highlight the following. 

\begin{corollary}\label{este}
Let $\{u_j\}$ be a bounded sequence in $NW^p(0, 1)$. Then there is a subsequence, not relabeled, and a function $u\in NW^{1, p}(0, 1)$ such that 
\begin{equation}\label{surconv}
u_j\to u\hbox{ in }L^p(0, 1),\quad \bbD u_j(x, X)\to\bbD u(x, X)\hbox{ for a.e. }(x, X)\in(0, 1)^2,
\end{equation}
and
$$
\bbD u_j(x, X)\rightharpoonup\bbD u(x, X) \hbox{ in }L^p((0, 1)^2).
$$
\end{corollary}
\begin{proof}
By Theorem \ref{uno}, there is a subsequence, not relabeled, such that 
$$
u_j\to u\hbox{ in }L^p(0, 1),\quad \bbD u_j\rightharpoonup\bU\hbox{ in }L^p((0, 1)^2),
$$
for some $u\in L^p(0, 1)$, and $\bU\in L^p((0, 1)^2)$. 
But the first convergence implies the pointwise convergence, possibly for a further subsequence, $\bbD u_j\to\bbD u$ in $(0, 1)^2$. Hence $\bbD u=\bU$,  $u\in NW^{1, p}(0, 1)$, and $\bbD u_j\rightharpoonup\bbD u$ in $L^p((0, 1)^2)$. 
\end{proof}

\begin{theorem}[Theorem 8.2, \cite{dinezzapalatuccivaldinoci}]
Every function in $NW^p(0, 1)$, for $p>2$, is H\"older continuous with exponent $\alpha=(p-2)/p$. In particular, end-point conditions on $\{0, 1\}$ for functions in these spaces are well-defined.
\end{theorem}

\section{Non-local variational problems in one-dimension}\label{cuatro}
The important conclusions in the last section lead to realizing that variational problems of the form
\begin{equation}\label{nlfuncional}
\hbox{Minimize in }u\in NW^{1, p}_0(0, 1):\quad E(u)=\int_0^1 \int_0^1 W(x, u(x), \bbD u(x, X))\,dX\,dx
\end{equation}
are meaningful under the usual polynomial coercivity condition
\begin{equation}\label{coercividadnl}
C_0(|U|^p-1)\le W(x, u, U),\quad C_0>0, p>2,
\end{equation}
for a density
$$
W(x, u, U):(0, 1)\times\R\times\R\to\R
$$
which is measurable in $x$ and continuous in $(u, U)$.
We have chosen, for the sake of definiteness, vanishing end-point conditions. That is what is meant, as one would expect, by $NW^{1, p}_0(0, 1)$ in \eqref{nlfuncional}. 

Minimizing sequences $\{u_j\}$ are uniformly bounded in $NW^{1, p}(0, 1)$. By Corollary \ref{este}, there is a limit feasible $u\in NW^p(0, 1)$ with
$u_j\to u$ in $L^p(0, 1)$, and 
\begin{equation}\label{convergenciapuntual}
\bbD u_j(x, X)\to\bbD u(x, X)
\end{equation}
for a.e. pair $(x, X)\in (0, 1)$. 
This a.e. convergence points in the direction of the following surprising result. Note that in this statement we are not assuming the lower bound \eqref{coercividadnl}. 
\begin{theorem}\label{wls}
Let the integrand
$$
W(x, u, U):(0, 1)\times\R\times\R\to\R
$$
be measurable in the variable $x$, continuous in pairs $(u, U)$, and bounded from below by some constant.
\begin{enumerate}
\item The corresponding functional $E(u)$ in \eqref{nlfuncional} is  weak lower semicontinuous in $NW^{1, p}(0, 1)$.
\item The same functional $E(u)$ is lower semicontinuous in $L^p(0, 1)$. 
\item If, in addition,
\begin{equation}\label{growthq}
|W(x, u, U)|\le C(1+|U|^q),\quad q<p,
\end{equation}
then $E(u)$ is weak continuous in $NW^{1, p}(0, 1)$.
\end{enumerate}
\end{theorem}

Note that there is no convexity  assumed on $W$.
\begin{proof}
The remarks above are already the basis for the proof, which is elementary at this point. 
The convergence $u_j\to u$ in $L^p(0, 1)$, implies the a.e. convergence \eqref{convergenciapuntual}.
Consequently, because of the continuity of $W$ with respect to variables $(u, U)$, 
$$
W(x, u_j(x), \bbD u_j(x, X))\to W(x, u(x), \bbD u(x, X))
$$
pointwise for a.e. $(x, X)\in(0, 1)^2$. If  $E(u_j)$ tends to infinity, there is nothing to be proved, as the conclusion is trivially true. If $\{E(u_j)\}$ is a bounded collection of numbers, the classical Fatou's lemma yields the claimed lower semicontinuity property
$$
E(u)\le\liminf_{j\to\infty} E(u_j).
$$
This covers the first two assertions. Concerning the third, just notice that the strict inequality in the previous argument with Fatou's lemma can only happen under concentration effects that are discarded, among other possible conditions, by a more restrictive growth condition on $W$ like \eqref{growthq}, if the weak convergence $u_j\rightharpoonup u$ takes place in $NW^{1, p}(0, 1)$.
\end{proof}

As a main consequence, we have a quite remarkable existence result for this kind of variational problems.
\begin{theorem}\label{existencia}
Consider problem \eqref{nlfuncional} for an integrand $W(x, u, U)$ which is measurable in $x$ and continuous in $(u, U)$, and satisfies \eqref{coercividadnl}. Suppose that the problem is not trivial ($E$ is finite for some feasible function). Then there are minimizers $u$ for \eqref{nlfuncional}, and minimizing sequences $\{u_j\}$ are such that 
\eqref{surconv} hold.
\end{theorem}
\begin{proof}
The proof is nothing but the direct application of the direct method to \eqref{nlfuncional}. 
\end{proof}

We cannot but conclude that both variational problems in Problems \ref{primero} and \ref{nlbolza} admit minimizers. We do not have any trouble accepting it for the former, but it is indeed a surprise for the latter. 

\section{Optimality}\label{cinco}
The study of optimality conditions for this kind of non-local variational problems lead to some unexpected answers too: optimality conditions are written in terms of integral equations, not differential equations. 

Let us place ourselves in a context where Theorem \ref{existencia} can be applied so that variational problem \eqref{nlfuncional} admits optimal solutions. 
Suppose that the integrand $W(x, u, U)$ is as smooth as we may need it to be for the calculations below to be valid.

Let $u\in NW^{1, p}(0, 1)$ be one such minimizer in a certain closed subspace of feasible functions in $NW^{1, p}(0, 1)$, and set $U\in NW^{1, \infty}_0(0, 1)$ for a feasible variation. As usual, the derivative of the section
$$
\epsilon\to \int_{(0, 1)^2}W(x, u(x)+\epsilon U(x), \bbD u(x, X)+\epsilon\bbD U(x, X))\,dX\,dx
$$
evaluated at $\epsilon=0$ must vanish. Since we are assuming whatever properties on $W$ for the derivation under the integral sign to be legitimate, we can write
\begin{equation}\label{optimalidad}
\int_{(0, 1)^2}[W_u(x, u(x), \bbD u(x, X))U(x)+W_U(x, u(x), \bbD u(x, X))\bbD U(x, X)]\,dX\,dx=0
\end{equation}
for all such $U(x)$. This is a well-defined double integral provided that
$$
|W_u(x, u, U)|\le C(1+|U|^{p-1}),\quad |W_U(x, u, U)|\le C(1+|U|^{p-1}).
$$
We examine the second term in this integral
$$
\int_{(0, 1)^2}\frac{W_U(x, u(x), \bbD u(x, X))}{X-x}(U(X)-U(x))\,dX\,dx.
$$
The inner single integrals
$$
\int_0^1\frac{W_U(x, u(x), \bbD u(x, X))}{X-x}\,dX
$$
for each fixed $x\in(0, 1)$, can be understood in a principal-value sense provided $W_U$ is 
continuous in all of its variables. Indeed, for $X$ near $x$, i.e. for $\epsilon$ small, 
$$
\int_{x-\epsilon}^{x+\epsilon}\frac{W_U(x, u(x), \bbD u(x, X))}{X-x}\,dX
$$
is approximately equal to
$$
W_U(x, u(x), u'(x))\int_{x-\epsilon}^{x+\epsilon}\frac1{X-x}\,dX=0.
$$
Hence, if we set
\begin{equation}\label{manipulacion}
\overline W(x, X)\equiv W_U(x, u(x), \bbD u(x, X)),
\end{equation}
and examine the integral
$$
\int_{(0, 1)^2}\overline W(x, X)\frac{U(X)-U(x)}{X-x}\,dX\,dx,
$$
which is the second full term in \eqref{optimalidad},
after a few simple, formal manipulations related to interchanging the order of integration, we find that the previous integral can be recast as
$$
-\int_{(0, 1)^2}\left[\frac{\overline W(X, x)+\overline W(x, X)}{X-x}\right]\,dX\, U(x)\,dx.
$$
If go back to \eqref{manipulacion}, and take back this fact to \eqref{optimalidad}, we end up with the condition
\begin{gather}
\int_0^1\int_0^1\left[W_u(x, u(x), \bbD u(x, X))+\right.\nonumber\\
\left.-\frac1{X-x}(W_U(x, u(x), \bbD u(x, X))+W_U(X, u(X), \bbD u(x, X)))\right]U(x)\,dX\,dx=0,\nonumber
\end{gather}
for every admissible variation $U\in NW^{1, \infty}_0(0, 1)$. Recall that $\bbD u(x, X)$ is symmetric. 
The arbitrariness of this test function $U$ leads to the condition
\begin{gather}
\int_0^1\left[W_u(x, u(x), \bbD u(x, X))+\right.\nonumber\\
\left.-\frac1{X-x}(W_U(x, u(x), \bbD u(x, X))+W_U(X, u(X), \bbD u(x, X)))\right]\,dX=0,\nonumber
\end{gather}
valid for a.e. $x\in(0, 1)$. For every such fixed $x\in(0, 1)$, these integrals should be understood in a principal-value sense, as indicated above, whenever necessary. 
The end-point conditions are irrelevant in these manipulations. 

Seeking some parallelism with the local case, we stick to the following definition following \cite{bellidocuetomora1}, \cite{ponce3}, \cite{shiehspector}, \cite{silhavy}.
\begin{definition}
For a measurable function 
$$
F(x, X):(0, 1)^2\to\R
$$
we define its non-local divergence as the function
$$
\ndv F(x, X)=\frac1{X-x}(F(x, X)+F(X, x)).
$$
\end{definition}
The previous manipulations show the following fact.
\begin{theorem}\label{inteq}
Let $W(x, u, U)$ be a $\cC^1$-integrand with respect to pairs $(u, U)$, such that
\begin{gather}
C_0(|U|^p-1)\le W(x, u, U)\le C(|U|^p+1), \nonumber\\
|W_u(x, u, U)|\le C(|U|^{p-1}+1),\quad |W_U(x, u, U)|\le C(|U|^{p-1}+1),\nonumber
\end{gather}
for some exponent $p>1$, and constants $0<C_0\le C$.
Suppose $u\in NW^{1, p}(0, 1)$ is a minimizer for  \eqref{nlfuncional} in a certain closed subspace of $NW^{1, p}(0, 1)$. Then
\begin{equation}\label{varinteq}
\int_0^1\left[-\ndv W_U(x, u(x), \bbD u(x, X))+W_u(x, u(x), \bbD u(x, X)\right]\,dX=0,
\end{equation}
for a.e. $x\in(0, 1)$, where the integrals of the first term should be understood in a principal-value sense whenever necessary.
\end{theorem}

To gain a bit of familiarity and realize what kind of integral equation these are, let us explore the form of this condition for the particular case in our Problem \ref{primero} in which, for the sake of simplicity in the computations, we take $p=2$ and 
$$
W(x, u, U)=\frac12U^2,\quad W_u=0,\quad W_U=U.
$$
The previous condition simplifies, after a few simple manipulations, to
\begin{equation}\label{inteqo}
\int_0^1\frac{u(X)-u(x)}{(X-x)^2}\,dX=0,\quad\hbox{ a.e. }x\in(0, 1).
\end{equation}
One would be tempted to separate the two integrals so as to write the condition as a more explicit integral equation. However, this separation is meaningless because the integral
$$
\int_0^1\frac1{(X-x)^2}\,dX
$$
is not finite for $X$ near $x$. 

Condition \eqref{inteqo} is definitely some sort of integral equation, but of a very special nature. In this form, no classic framework in the field of Integral Equations (\cite{zemyan}) seems to match \eqref{inteqo}. One thing is however clear: the admissible function $u(x)=x$ cannot be a minimizer because the integral
$$
\int_0^1\frac1{X-x}\,dX
$$
does not vanish for every $x\in(0, 1)$. This is elementary to check. 

For Problem \ref{nlbolza}, we find the impressive integral equation, after a few algebraic manipulations,
$$
\frac{u(x)}2=\int_0^1\frac{u(X)-u(x)}{(X-x)^2}\left[\frac{(u(X)-u(x))^2}{(X-x)^2}-1\right]\,dX.
$$
Note that the trivial function $u\equiv0$ is a solution.

\section{Integral equations}\label{seis}
The classical theory of Integral Equations in one independent variable (\cite{zemyan}) focuses on functional equations of the form
\begin{equation}\label{geninteq}
h(x)u(x)=f(x)+\int_a^{b(x)}K(x, X)u(X)\,dX,
\end{equation}
for functions $h(x)$, $f(x)$, and $b(x)$. $a$ is a real number, and $K(x, X)$ is the kernel of the equation. The nature of the three functions $h$, $f$ and $b$, and the properties of the kernel $K$ determine the type of equation (homogeneous/non-homogeneous, Fredholm, Voterra, of the first/second kind, etc), and, eventually, its understanding and potential methods of solution. It is not clear how an integral equation of the form in Theorem \ref{inteq} could be recast to fit the form \eqref{geninteq}. 

\begin{definition}
An integral equation is called variational if there is a $\cC^1$-function
$$
W(x, u, U):(0, 1)\times\R\times\R\to\R,
$$
with continuous partial derivatives $W_u(x, u, U)$ and $W_U(x, u, U)$, such that the integral equation is written in the form \eqref{varinteq}.
\end{definition}
We can translate Theorems \ref{existencia} and \ref{inteq} into an existence theorem for this kind of integral equations.
\begin{theorem}
Let 
$$
W(x, u, U):(0, 1)\times\R\times\R\to\R
$$
be a $\cC^1$-function in pairs $(u, U)$ such that
$$
C_0(|U|^p-1)\le W(x, u, U)\le C(|U|^p+1),\quad C\ge C_0>0, p>2,
$$
and
$$
|W_u(x, u, U)|\le C(|U|^{p-1}+1),\quad |W_U(x, u, U)|\le C(|U|^{p-1}+1).
$$
Then for arbitrary end-point conditions 
$$
u(0)=u_0,\quad u(1)=u_1,
$$ 
the variational integral equation
$$
\int_0^1\left[-\ndv W_U(x, u(x), \bbD u(x, X))+W_u(x, u(x), \bbD u(x, X)\right]\,dX=0
$$
for a.e. $x\in(0, 1)$ admits solutions. 
\end{theorem}

We go back to our basic example \eqref{inteqo}  to perform some simple formal manipulations, again taking $p=2$. The pecularities of such an integral equation make it impossible to follow some of the methods that are used for more standard integral equations (\cite{zemyan}). In particular, integral transform techniques seem out of context as the interval of integration is finite, while the reduction to some kind of differential equation by direct differentiation with respect to the variable $x$ looks hopeless too. If we are contented with some sort of approximation, then we can play with it in several ways. 
It is legitimate a first integration by parts to find
$$
\int_0^1\frac{u'(X)}{X-x}\,dX=\frac1{1-x}-\frac1{x(1-x)}u(x),
$$
or even better
\begin{equation}\label{inteqo2}
\int_0^1\frac{x(1-x)}{X-x}u'(X)\,dX=x-u(x).
\end{equation}
The integral in the left-hand side ought to be understood in a principal-value sense. If we put
$$
u'(X)=v(X),\quad \int_0^1 v(X)\,dX=1,
$$
then, for the kernel
$$
K(x, X)=\frac{x(1-x)}{X-x}+\chi_{(0, x)}(X),
$$
where $\chi_{(0, x)}(X)$ is the indicator function of the interval $(0, x)$, then \eqref{inteqo2} becomes
$$
\int_0^1 K(x, X)v(X)\,dX=x.
$$

To find some approximation of the function we are searching for, let us go back to \eqref{inteqo}, and write the approximation
$$
u(X)-u(x)\sim u'(x)(X-x)+\frac12u''(x)(X-x)^2.
$$
Then \eqref{inteqo} becomes
$$
u'(x)\int_0^1\frac1{X-x}\,dX+\frac12u''(x)\sim0\hbox{ in }(0, 1),
$$
where the integral is again interpreted in a principal-value sense. We are led to consider the second-order ODE
$$
\log\frac{1-x}xu'(x)+\frac12u''(x)=0\hbox{ in }(0, 1),
$$
which after some elementary manipulations is transformed into
$$
u'(x)=kx^{2x}(1-x)^{2(1-x)},\quad x\in(0, 1),
$$
where the constant $k$ is chosen so that
$$
k^{-1}=\int_0^1x^{2x}(1-x)^{2(1-x)}\,dx.
$$
Check this profile in Figure \ref{boceto} for $k=2$. 
\begin{figure}[b]
\includegraphics[scale=0.35]{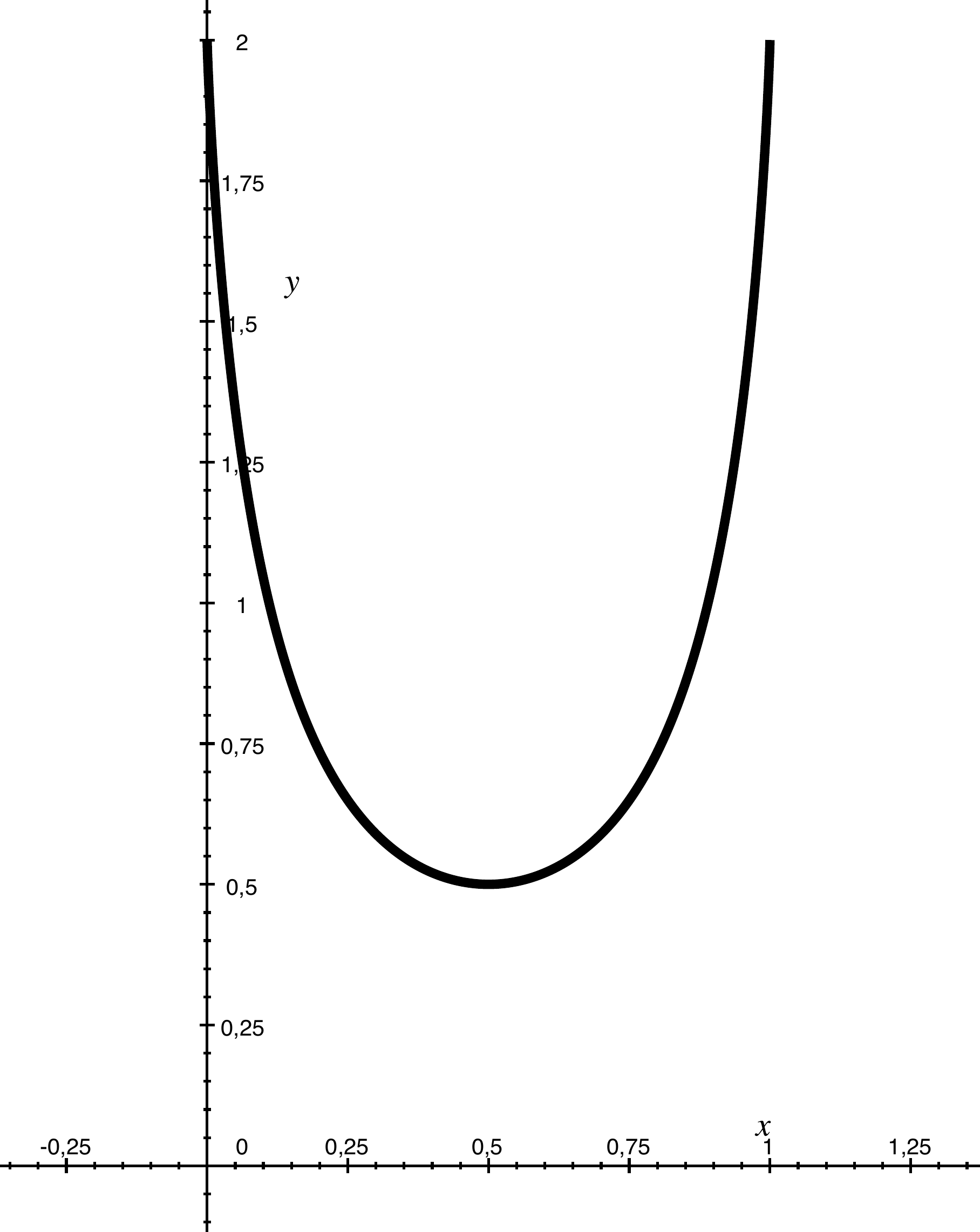}
\caption{An approximation of the derivative of the optimal solution for the classical quadratic, homogeneous case.} \label{boceto}
\end{figure}

\section{Approximation of optimal solution for simple examples}
Even though our existence Theorem \ref{existencia} yields optimal solution for non-local variational problems of the kind considered here, when the integrand is not (strictly) convex one misses three main points: uniqueness, sufficiency of optimality conditions, and reliable numerical approximation. One can hardly rely on numerical calculations for the optimal solutions of Problem \ref{nlbolza}, but one can go through simple approximation schemes for convex problems. 

For the sake of illustration, we show results for Problem \ref{primero} for the exponent $p=2$, and some easy variation.
\begin{enumerate}
\item The unique optimal profile for Problem \ref{primero} is depicted in Figure \ref{primeroo}. Note how, qualitatively, its derivatives yields the graph in Figure \ref{boceto}. 
\begin{figure}[b]
\includegraphics[scale=0.25]{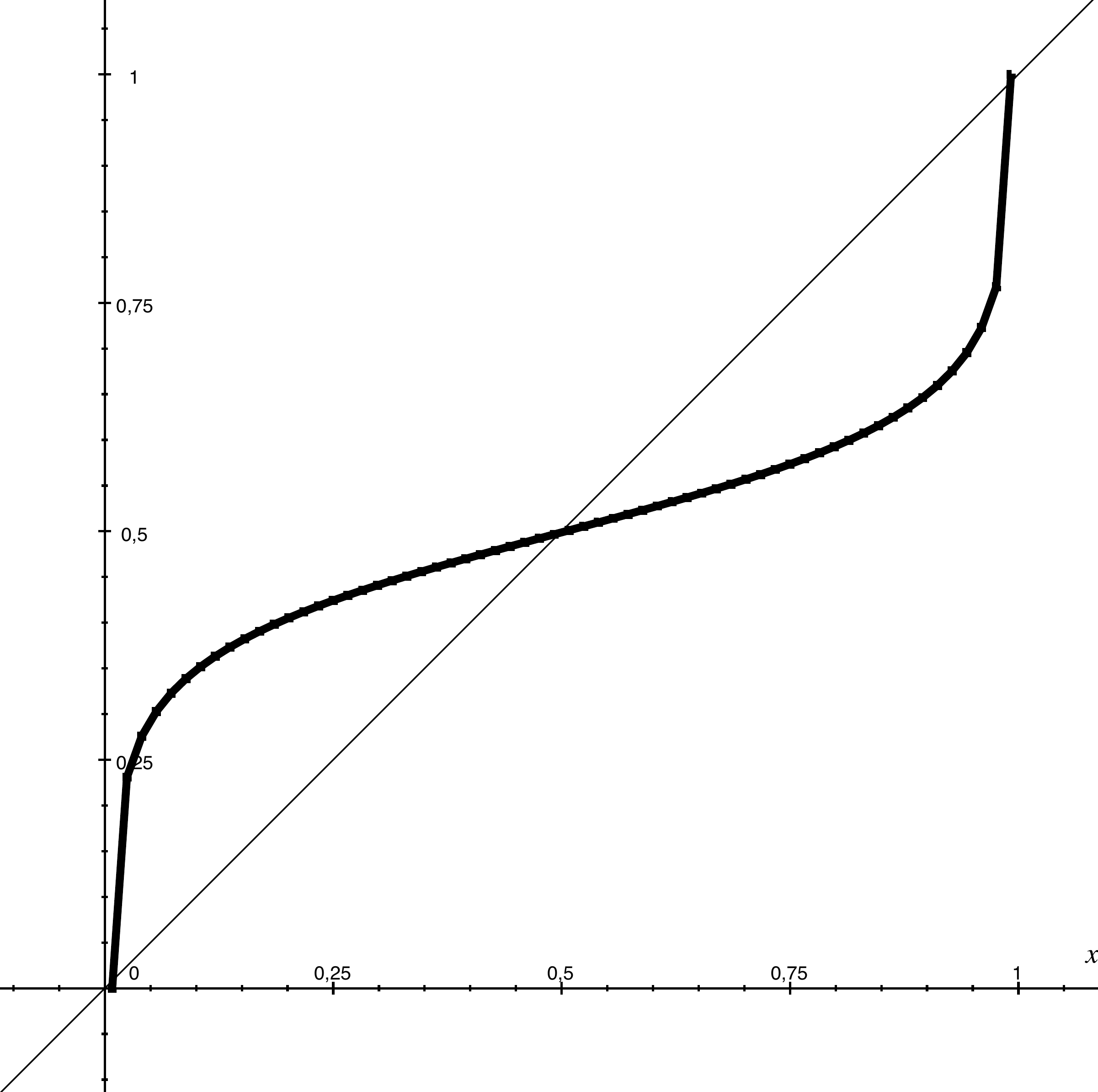}
\caption{The classical quadratic, homogeneous case.} \label{primeroo}
\end{figure}

\item We look at the problem
$$
E(u)=\int_0^1\int_0^1\frac12\left(\frac{u(x)-u(y)}{x-y}\right)^2\,dx\,dy+8\int_0^1u(x)^2\,dx
$$
again under end-point conditions $u(0)=0$, $u(1)=1$. The unique solution for the corresponding local problem
$$
I(u)=\int_0^1\left[\frac12 u'(x)^2+8u(x)^2\right]\,dx
$$
is 
$$
u(x)=\frac{e^4}{e^8-1}(e^{4x}-e^{-4x}).
$$
Both are compared in Figure \ref{segundo}. 
\begin{figure}[b]
\includegraphics[scale=0.4]{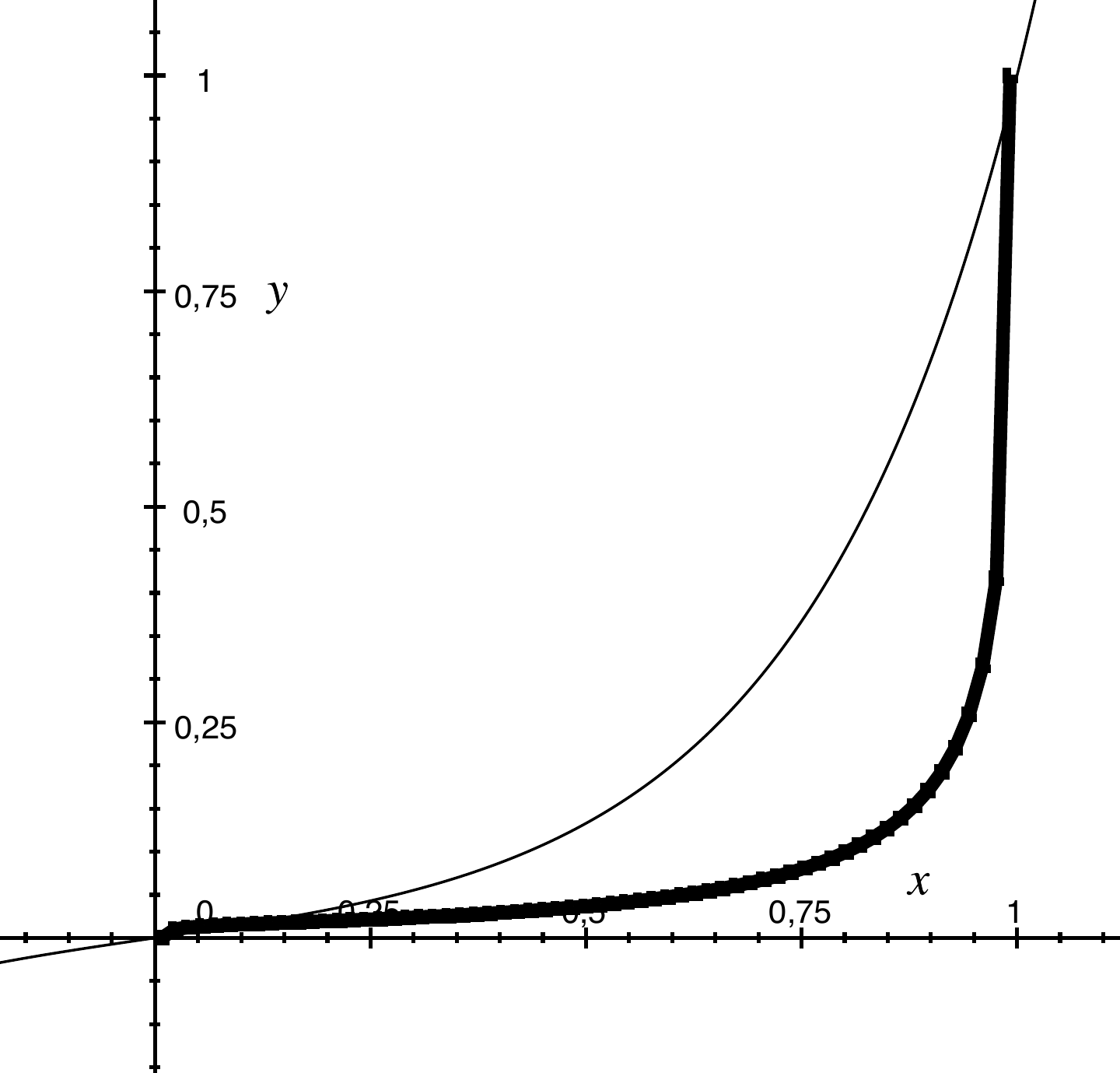}
\caption{A variant of the quadratic case.} \label{segundo}
\end{figure}
\item As indicated above, it is not possible to perform reliable numerical calculations for the non-convex case Problem \ref{nlbolza}, either with the lower-order term or without it. Check Figure \ref{tercero} for a couple of simulations for a functional without the lower-order term, starting from the trivial map. The difference of the two picture is in the discretization used: the one on the right used as much as twice elements than the one on the left, and yet the computations were unable to produce finer oscillations. The two drawings are indistinguishable. This fact has to be taken with extreme caution. What is true is that, according to our Theorem \ref{existencia}, there are minimizers for such a non-convex problem which, presumably, would show a certain finite number of oscillations. This is also true for the functional with the lower-order contribution. 
\begin{figure}[b]
\includegraphics[scale=0.2]{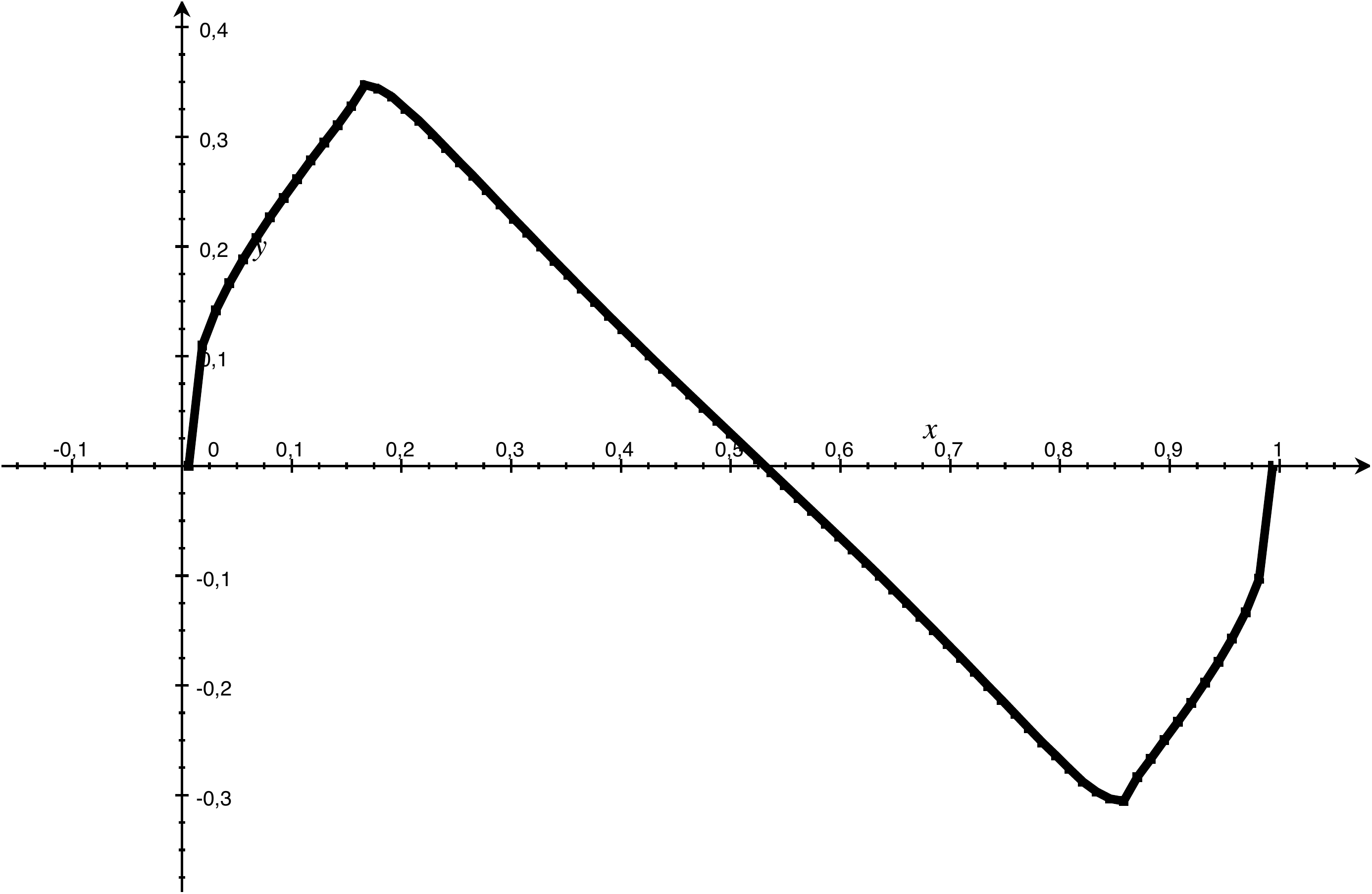}\quad 
\includegraphics[scale=0.2]{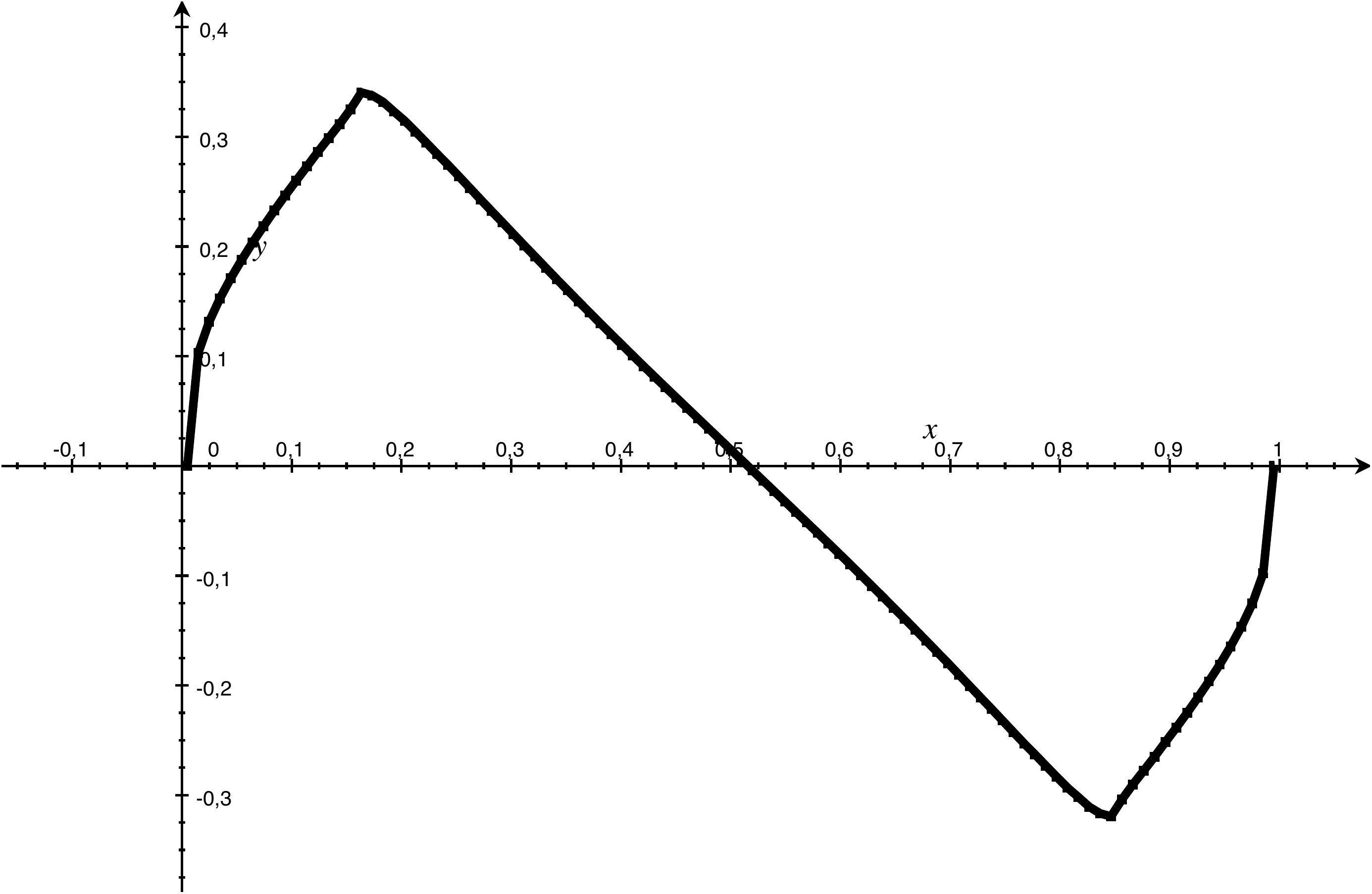}
\caption{The non-convex case.} \label{tercero}
\end{figure}
\end{enumerate}

\end{document}